\documentclass{amsart}
\usepackage{amsmath, amssymb, mathrsfs, verbatim, multirow}
\usepackage[pdftex,
            pdfauthor={M. Ghasemi, J. B. Lasserre, M. Marshall},
            pdftitle={Lower bounds on the global minimum of a polynomial},
            pdfsubject={Lower bounds on the global minimum of a polynomial},
            pdfkeywords={Optimization, Global, Constrained, Semidefinite programming, Geometric Programming},
            pdfproducer={Latex with hyperref},
            pdfcreator={PDFLaTeX},
            colorlinks=true]{hyperref}
\newcommand{\mathsym}[1]{{}}

\newtheorem{thm}{Theorem}[section]
\newtheorem{lemma}[thm]{Lemma}
\newtheorem{proposition}[thm]{Proposition}

\theoremstyle{definition}

\newtheorem{exm}[thm]{Example}

\newtheorem{rem}[thm]{Remark}

\def\beq{\begin{equation} }
\def\eeq{\end{equation} }

\def\N{\mathbb{N}}

\def\R{\mathbb{R}}

\def\P{\mathbf{P}}

\def\u{\mathbf{u}}

\def\x{\mathbf{x}}

\def\z{\mathbf{z}}

\def\dis{\displaystyle}

\def\R{\mathbb{R}}

\def\dis{\displaystyle}

\begin{document}
\title{Lower bounds on the global minimum of a polynomial}
\author{M. Ghasemi, J.B. Lasserre, M. Marshall}

\begin{abstract}
We extend the
method of Ghasemi and Marshall [SIAM. J. Opt. 22(2) (2012), pp 460-473], to obtain a lower bound $f_{{\rm gp},M}$
for a multivariate polynomial $f(\x) \in \R[\x]$ of degree $ \le 2d$ in $n$ variables $\x = (x_1,\dots,x_n)$ on the closed ball $\{ \x \in \R^n : \sum x_i^{2d} \le M\}$, computable by geometric programming, for any real $M$. We compare this bound with the (global) lower bound $f_{{\rm gp}}$
obtained by Ghasemi and Marshall, and also with the hierarchy of lower bounds, computable by semidefinite programming, obtained by Lasserre [SIAM J. Opt. 11(3) (2001) pp 796-816]. Our computations show that the bound $f_{{\rm gp},M}$ improves on the bound $f_{{\rm gp}}$ and that the computation of $f_{{\rm gp},M}$, like that of $f_{{\rm gp}}$, can be carried out quickly and easily for polynomials having of large number of variables and/or large degree, assuming a reasonable sparsity of coefficients, cases where the corresponding computation using semidefinite programming breaks down.
\end{abstract}


\maketitle

\section{Introduction}

Computing a lower bound on the global minimum on $\R^n$ of a multivariate polynomial is
a standard problem of optimization with many potential applications. In the last decade,
results in polynomial optimization combined with semidefinite programming (for sums of squares representation),
have permitted to make some progress. For instance, one may compute a lower bound of  polynomial $f\in\R[\x]$ on $\R^n$:
\begin{itemize}
\item by solving the problem $f_{{\rm sos}}:=\sup\{\lambda: f-\lambda \mbox{ is sos}\}$, which is a single semidefinite program
\item by applying the hierarchy of semidefinite relaxations
to the polynomial optimization problem $\inf \{f(\x): \nabla f(\x)=0\}$ (assuming that the infimum is attained)
\item by applying the hierarchy of semidefinite relaxations
to the polynomial optimization problem $\inf \{f(\x): \Vert\x\Vert^{2}\leq M\}$, for sufficiently large $M$
(assuming that a global minimum satisfies that bound constraint).
\end{itemize}
All those approaches are very powerful and provide good bounds and sometimes the exact
value. However, so far, and in view of the present status of semidefinite programming, those methods are  limited to
small to medium size problems, except if some structured sparsity is present (in which case specialized semidefinite relaxations can be implemented; see e.g. \cite{kim}).

This limitation of semidefinite programming to implement sums of squares (SOS) representations,
was the motivation for providing other SOS certificates and yielded the sufficient conditions of
\cite{lasserre} and subsequently of \cite{fidalgo,gha-mar1}.
And in a recent work Ghasemi and Marshall \cite{gha-mar2} have shown how to compute a lower bound on
the global optimum of a multivariate polynomial on $\R^n$, by solving a certain geometric program.
This formulation as a geometric program is based on the sufficient condition for a polynomial to be a
sum of squares given in \cite{gha-mar1}, which generalizes the sufficient conditions of
\cite{fidalgo,lasserre}. Geometric programming (GP) is a convex optimization problem
that can be solved efficiently for relatively large scale problems. In Boyd et al.
\cite{boyd} it is claimed that GP problems with up to
$10^3$ variables and $10^4$ constraints can be solved via standard interior point methods.
For sparse GP problems, i.e., where each constraint depends only on a small number of variables,
the size limit can grow up to $10^4$ variables and $10^6$ constraints!
So the interest of the geometric programming formulation is
that one may now handle polynomials with a large number of variables and high degree, especially when the support of $f$
(i.e., the set  of non zero coefficients) is small.

{\bf Contribution.} Our contribution is to extend the geometric programming formulation of
Ghasemi and Marshall \cite{gha-mar2} to provide a lower bound on
$f_{*,M}:=\min \{f(\x):\sum_i x_i^{2d}\leq M\}$. The latter problem has its own interest and also serves as an auxiliary problem to provide a lower bound on $f_* :=\min \{f(\x):\x\in\R^n\}$ when a global minimizer is ``guessed" to belong
to the ball $\{\x:\sum_ix_i^{2d}\leq M\}$. Again, and as for \cite{gha-mar2}, the main interest of this approach is
to be able to handle polynomials with large number of variables and/or large degree for which so far, there is no such algorithm.
Notice that even for a small number of variables, the SOS approaches cannot handle polynomials with large degree.

\section{Main result}

\subsection*{Notation and definitions}

Let $\R[\x]$ be the ring of polynomials in the variables $\x=(x_1,\ldots,x_n)$, and for $d\in\N$,
let $\R[\x]_d\subset\R[\x]$ be
the vector space of polynomials of degree at most $d$. Let
$\N^n_d:=\{\alpha\in\N^n:\vert\alpha\vert\leq d\}$ where $\vert\alpha\vert=\sum_i\alpha_i$ for every $\alpha\in\N^n$.

Assume now that $d\ge 1$. Let $\epsilon_i:=(\delta_{i1},\cdots,\delta_{in})\in\N^n$, with $\delta_{ij}=1$ if $i=j$ and $0$ otherwise, and
given $f= \sum f_{\alpha} \x^{\alpha}\in\R[\x]_{2d}$,  let:
\begin{eqnarray*}
\Omega(f)&:=&\{\alpha\in\N^n_{2d}\::\:f_\alpha\neq0\}\setminus\{0,2d\epsilon_1,\cdots,2d\epsilon_n\}\\
\Delta(f)&:=&\{\alpha\in\Omega(f)\::\:f_\alpha\,\x^\alpha\mbox{ is not a square in $\R[\x]$}\}\\
\Delta(f)^{< 2d}&:=&\{\alpha\in\Delta(f)\::\:\vert\alpha\vert <2d\}.
\end{eqnarray*}
Denote the coefficient $f_{2d\epsilon_i}$ by $f_{2d,i}$ for $i=1,\dots,n$.

We first recall the following result of Ghasemi and Marshall \cite{gha-mar2}.

\begin{proposition}(\cite[Corollary 3.6]{gha-mar2})
\label{prop1}
Let $f\in\R[\x]_{2d}$ and let $\rho$ be the optimal value of the program:
\begin{equation}
\label{lw}
\left\{\begin{array}{rl}
\rho=\dis\min_{\z_\alpha}&\dis\sum_{\alpha\in\Delta(f)^{<2d}}(2d-\vert\alpha\vert)\left[
\left(\frac{f_\alpha}{2d}\right)^{2d}\,\left(\frac{\alpha}{\z_\alpha}\right)^\alpha
\right]^{1/(2d-\vert\alpha\vert)}\\
&\\
\mbox{s.t.}&\dis\sum_{\alpha\in\Delta(f)}z_{\alpha,i \ \le \ f_{2d,i}},\quad i=1,\ldots,n\\
&\\
&\left(\frac{2d}{f_\alpha}\right)^{2d}\,\left(\frac{\z_\alpha}{\alpha}\right)^\alpha \,=\,1;\quad\alpha\in\Delta(f),\:\vert\alpha\vert=2d.
\end{array}\right.\end{equation}
where for every $\alpha\in\Delta(f)$, the unknowns $\z_\alpha = (z_{\alpha,i})\in\R^n_+$ satisfy $z_{\alpha,i}=0$ if and only if $\alpha_i=0$. Here, $\left(\frac{\alpha}{\z_\alpha}\right)^\alpha := \prod_{i=1}^n \frac{\alpha_i^{\alpha_i}}{(z_{d,i})^{\alpha_i}}$ and $\left(\frac{\z_\alpha}{\alpha}\right)^\alpha := \prod_{i=1}^n \frac{(z_{d,i})^{\alpha_i}}{\alpha_i^{\alpha_i}}$ with the convention $0^0 = 1$.
Then $f(\x)\geq f(0)-\rho$ for all $\x\in\R^n$.\end{proposition}

The most interesting case is when $f_{2d,i} >0$, $i =1,\dots,n$, in which case the program (\ref{lw}) is a geometric program. Somewhat more generally, if $\forall$ $i=1,\dots,n$ either ($f_{2d,i}>0$) or ($f_{2d,i}=0$ and  $\alpha_i=0$ $\forall$ $\alpha \in \Delta(f)$), then the program (\ref{lw}) is a geometric program. In the remaining cases the program
(\ref{lw}) is not a geometric program, the feasibility set of (\ref{lw}) is empty, and the output $\rho$ is $\infty$.

\subsection*{Problem statement}

Let  $f\in\R[\x]$ and, for $M>0$, consider the problem:
\begin{equation}
\label{pb-M}
\P_M:\quad f_{*,M}:=\min \:\{f(\x)\::\:\sum_{i=1}^n x_i^{2d}\leq\,M\}.\end{equation}
Problem $\P_M$ has its own interest but is also an auxiliary problem for the unconstrained problem
$\P_\infty:\: f_*=\min\{f(\x):\x\in\R^n\}$, when a global minimizer is guessed to belong to the ball
$B_M:= \{\x:\sum_ix_i^{2d}\leq M\}$. Also, notice that the sequence $(f_{*,M})$, $M\in\N$,  provides a monotone nonincreasing sequence of upper bounds on $f_*$ that converges to $f_*$ in finitely many steps whenever $\P_{\infty}$ has an optimal solution $\x^*\in\R^n$.

\subsection*{Main result}

With $M>0$ fixed, to compute a lower bound for $f_{*,M}$, let $\lambda\geq0$ and
consider the polynomial $f_\lambda\in\R[\x]$
\begin{equation}
\label{def1}
\x\mapsto f_\lambda(\x)\,=\,f(\x)-\lambda(M-\sum_{i=1}^nx_i^{2d}),\qquad\lambda\geq0.\end{equation}
%Then :
\begin{lemma}
Let $f\in\R[\x]$, $\deg \, f \le 2d$ and let $f_\lambda\in\R[\x]$ be as in (\ref{def1}). Then:
\begin{equation}
\label{lag1}
f_{*,M}\,\geq\,\dis\max_{\lambda\geq0}\:\dis\underbrace{\min_{\x\in\R^n} f_\lambda(\x)}_{G(\lambda)}\,=\,\dis\max_{\lambda\geq0}\:G(\lambda).
\end{equation}
Moreover, if either $f_*= f_{*,M}$ or $f$ is convex then equality holds.
\end{lemma}
The proof is standard and will be omitted.
Actually, one can show that
$$\dis\max_{\lambda\geq0}\:\dis\min_{\x\in\R^n} f_\lambda(\x)\: = \:\dis\min_{\x\in\R^n} f_{\lambda_1}(\x) \: = (f_{\lambda_1})_{*,M}$$
where
$\lambda_1$ is the least $\lambda \ge 0$ such that $f_{\lambda}$ achieves its global minimum on the ball $B_M$.
Obviously, $f_{*,M} \ge (f_{\lambda_1})_{*,M}$. If $f_*= f_{*,M}$ then $\gamma_1 = 0$ and $f_{*,M} = (f_{\lambda_1})_{*,M}$. If $f$ is convex then $f_{\gamma}$ is convex for each $\gamma \ge 0$. If $f$ is convex and  $\gamma_1 >0$ then the minimum of $f_{\gamma}$ on $B_M$ is achieved on the boundary of $B_M$ for $0\le \gamma \le \gamma_1$, so $f_{*,M} = (f_{\lambda_1})_{*,M}$ holds in this case too.

Note that equality in (\ref{lag1}) fails in general.

\begin{exm} Let $n=1$, $2d=4$, $f(x) = 2x^2(x-2)^2+(1-x^4) = x^4-8x^3+8x^2+1$, $M=1$. Then $f_{*,M} = 1$ and $\lambda_1 = 1$ so  $\dis\max_{\lambda\geq0}\:\dis\min_{x\in\R} f_\lambda(x)= \dis\min_{x\in\R} f_{\lambda_1}(x)\:= 0$.
\end{exm}

Observe that for every $\lambda\geq0$,
\begin{equation}
\label{Glambda}
G(\lambda)\,=\,\min_{\x\in\R^n} f_\lambda(\x),\end{equation}
and so if for every $\lambda\geq0$,
$\overline{G}(\lambda)$ is a lower bound on $G(\lambda)$, then
\begin{equation}
\label{lag2}
f_{*,M}\,\geq\,
\dis\max_{\lambda\geq0}\:G(\lambda)\,\geq\,\dis\max_{\lambda\geq0}\:\overline{G}(\lambda).\end{equation}

After relabeling if necessary, we may and will assume that
\[f_{2d,1}\,\geq\,f_{2d,2}\,\geq\,\cdots \geq\,f_{2d,n}.\]
The main result of our paper is as follows:
\begin{thm}
\label{thmain}
Let $f\in\R[\x]$, ${\rm deg}\,f \le 2d$. Then
\[f_{*,M}\,\geq\, f(0)+M f_{2d,1}-\rho_M,\]
with $\rho_M$ being the optimal value of the geometric program:
\begin{equation}
\label{newgp}
\left\{\begin{array}{cl}
\rho_M=\dis\min_{\z_\alpha,\u}&Mu_1+\dis\sum_{\alpha\in\Delta(f)^{<2d}}(2d-\vert\alpha\vert)\left[
\left(\frac{f_\alpha}{2d}\right)^{2d}\,\left(\frac{\alpha}{\z_\alpha}\right)^\alpha
\right]^{1/(2d-\vert\alpha\vert)}\\
&\\
\mbox{s.t.}&\dis\sum_{\alpha\in\Delta(f)}\frac{z_{\alpha,i}}{u_i}\leq1,\quad i=1,\ldots,n\\
&\\
&\left(\frac{2d}{f_\alpha}\right)^{2d}\,\left(\frac{\z_\alpha}{\alpha}\right)^\alpha \,=\,1;\quad\alpha\in\Delta(f),\:\vert\alpha\vert=2d.\\
&\\
(*)&\frac{f_{2d,1}}{u_1}\,\leq\,1\\
&\\
(**)&\frac{u_i}{u_{i-1}}+\frac{f_{2d,i-1}-f_{2d,i}}{u_{i-1}}\,\leq\,1,\quad i=2,\ldots,n,
\end{array}\right.\end{equation}
and where for every $\alpha\in\Delta(f)$, the unknowns $\z_\alpha=(z_{\alpha,i})\in\R^n_+$ satisfy $z_{\alpha,i}=0$ if and only if $\alpha_i=0$.
\end{thm}
A detailed proof can be found in \S \ref{proof}. Observe that the difference between the programs
(\ref{lw}) and (\ref{newgp}) is the presence of the constraints $(*)-(**)$ in the latter, which reflects the new contribution of the
monomial terms $\lambda x_i^{2d}$ in the polynomial $f_\lambda$.

The geometric program (\ref{newgp}) is {\it not} a direct application of Proposition
\ref{prop1} to the polynomial $f_\lambda$  to obtain a lower bound $\overline{G}(\lambda)$ on $G(\lambda)$, followed by a maximization with respect to $\lambda$.
Indeed, this leads to the constraint
$(**)$ in {\it equality} (instead of inequality) form, and so  (\ref{newgp}) would not be a geometric program; however, in the proof we show that this equality constraint can be relaxed to an inequality constraint as in (\ref{newgp}).

\section{Comparison with other bounds}

\subsection*{Comparison with bound of Ghasemi and Marshall} Assume that $f \in \R[\x]_{2d}$, $d\ge 1$.  As in \cite{gha-mar2} we define $f_{{\rm gp}}$ to be $f_{{\rm gp}}:=f(0)-\rho$, the lower bound for $f_*$ obtained in Proposition \ref{prop1}. We also define $f_{{\rm gp},M}$ to be $f_{{\rm gp},M}:=f(0)+Mf_{2d,1}-\rho_M$, the lower bound for $f_{*,M}$ obtained in Theorem \ref{thmain}.
Note that the feasible set of (\ref{newgp}) is nonempty (i.e., $f_{{\rm gp},M}$ is a real number), whereas the feasible set of  (\ref{lw}) may be empty (i.e., $f_{{\rm gp}}=-\infty$), even in the case where each $f_{2d,i}$ is strictly positive.

\begin{proposition} \label{p2} \
\begin{enumerate}
\item $f_{{\rm gp},M} \ge f_{{\rm gp}}$.
\item If $M'\le M$  then $f_{{\rm gp},M'} \ge f_{{\rm gp},M}$.
\item $f_{{\rm gp}} = \lim\limits_{M\rightarrow \infty}f_{{\rm gp},M}$.
\end{enumerate}
\end{proposition}

\begin{proof}
(1) If the program (\ref{lw}) in Proposition \ref{prop1} has no feasible solutions then $f_{{\rm gp}}=-\infty$ so $f_{{\rm gp},M} \ge f_{{\rm gp}}$. Suppose now that (\ref{lw}) has a feasible solution $\z$. In particular, $f_{2d,i}\ge 0$ for $i=1,\dots,n$. Fix $\delta >0$. Then $(\z,\u)$ with $u_i=f_{2d,i}+\delta$ for all $i=1,\ldots,n$, is feasible for the program
(\ref{newgp}) in Theorem \ref{thmain}. This implies $\rho_M\leq M(f_{2d,1}+\delta)+\rho$ for all $\delta>0$ so $\rho_M\leq Mf_{2d,1}+\rho$ and
\begin{eqnarray*}
f_{{\rm gp},M}&=&f(0)+Mf_{2d,1}-\rho_M\\
&\geq& f(0)+Mf_{2d,1}-Mf_{2d,1}-\rho\,=\,f(0)-\rho\,=\,f_{{\rm gp}}.\end{eqnarray*}

(2) Suppose $M'\le M$. Observe that the set of feasible solutions for (\ref{newgp}) does not depend on $M$. Let $(\z,\u)$ be a feasible solution of (\ref{newgp}). Since $M'<M$ and $u_1\ge f_{2d,1}$ it follows that $M'(u_1-f_{2d,1})\le M(u_1-f_{2d,1})$. This implies that $\rho_{M'} -M'f_{2d,1} \le \rho_M-Mf_{2d,1}$, so $$f_{{\rm gp},M'} = f(0)+M'f_{2d,1}-\rho_{M'} \ge f(0)+Mf_{2d,1}-\rho_M = f_{{\rm gp},M}.$$

(3) It remains to show that if there exists a real number $N$ such that $\rho_M-Mf_{2d,1}\le N$ for each real $M>0$ then $\rho \le N$. Suppose $\rho_M-Mf_{2d,1} \le N$ for all $M>0$. Then for each real $\epsilon >0$ there exists a feasible solution $(\z,\u) = (\z_M,\u_M)$ of (\ref{newgp}) such that \begin{equation}
\label{P}
M(u_1-f_{2d,1})+\dis\sum_{\alpha\in\Delta(f)^{<2d}}(2d-\vert\alpha\vert)\left[
\left(\frac{f_\alpha}{2d}\right)^{2d}\,\left(\frac{\alpha}{\z_\alpha}\right)^\alpha
\right]^{1/(2d-\vert\alpha\vert)}\le N+\epsilon,
\end{equation}
for $M=1,2,\cdots$. As explained in the proof of Theorem \ref{thmain}, we may assume $u_i-f_{2d,i}= u_j-f_{2d,j}$ for all $i,j = 1,\dots,n$. Let $\lambda = \lambda_M = u_1-f_{2d,1}$, so $u_i=f_{2d,i}+\lambda$ for $i=1,\dots,n$. From inequality (\ref{P}) we see that $M\lambda\le N+\epsilon$, so $\lambda \rightarrow 0$ as $M \rightarrow \infty$. Since $0\le z_{\alpha,i} \le u_i =f_{2d,i}+\lambda$, the sequence $(\z,\u) = (\z_M,\u_M)$ is bounded so it has some convergent subsequence converging to some $(\z^*,\u^*)$. If $\z^*$ is a feasible point of the program (\ref{lw}) then we see by continuity that $$\dis\sum_{\alpha\in\Delta(f)^{<2d}}(2d-\vert\alpha\vert)\left[
\left(\frac{f_\alpha}{2d}\right)^{2d}\,\left(\frac{\alpha}{{\z^*}_\alpha}\right)^\alpha
\right]^{1/(2d-\vert\alpha\vert)}\le N+\epsilon$$ so $\rho \le N+\epsilon$ and we are done. The fact that $\z^*$ is a feasible point for (\ref{lw}) is more or less clear, by continuity, except possibly for the fact that $\alpha_i >0$ $\Rightarrow$ ${z^*}_{\alpha,i}>0$. If $|\alpha|=2d$ this follows from the equation $\left(\frac{2d}{f_\alpha}\right)^{2d}\,\left(\frac{\z_\alpha}{\alpha}\right)^\alpha \,=\,1$ which, since the $z_{\alpha,i}$ are bounded, implies that the $z_{\alpha,i}$ such that $\alpha_i>0$ are bounded away from zero. Similarly for $|\alpha|<2d$ the inequality (\ref{P}) implies that the $z_{\alpha,i}$ such that $\alpha_i>0$ are bounded away from zero.
\end{proof}

\subsection*{Comparison with bounds of Lasserre} Recall that \[f_{{\rm sos}}:=\sup\{\lambda: f-\lambda \in \sum \R[\x]^2\}.\] The inequality $f_* \ge f_{{\rm sos}}$ is trivial. The inequality $f_{{\rm sos}} \ge f_{{\rm gp}}$ is established in \cite[Corollary 3.6]{gha-mar2}. As explained in \cite{lasserre1}, $f_{{\rm sos}}$ is computable by semidefinite programming. Similarly, for each real $M>0$ and each integer $k\ge0$ define $f_{{\rm sos},M}^{(k)}$ to be the supremum of all real numbers $\lambda$ such that $$f-\lambda = \sigma+\tau(M-\sum x_i^{2d})$$ for some $\sigma,\tau \in \sum \R[\x]^2$, $\deg(\sigma)\le 2k+2d$, $\deg(\tau)\le 2k$. As explained in \cite{lasserre1}, the sequence $f_{{\rm sos},M}^{(k)}$, $k=0,1,\cdots$ is nondecreasing and converges to $f_{*,M}$ as $k \rightarrow \infty$ and each $f_{{\rm sos},M}^{(k)}$ is computable by semidefinite programming.

\begin{proposition} $f_{{\rm sos},M}^{(0)} \ge f_{{\rm gp},M}$.
\end{proposition}

\begin{proof} By the proof of Theorem \ref{thmain}, $f_{{\rm gp},M} = \max\limits_{\lambda \ge 0} \overline{G}(\lambda)$ where $\overline{G}(\lambda):= (f_{\lambda})_{{\rm gp}}$. By \cite[Corollary 3.6]{gha-mar2}, $(f_{\lambda})_{{\rm sos}} \ge (f_{\lambda})_{{\rm gp}}$. Thus for any real $\epsilon >0$ there exists $\lambda \ge 0$ such that $(f_{\lambda})_{{\rm sos}} \ge (f_{\lambda})_{{\rm gp}} \ge f_{{\rm gp},M}-\epsilon$, so there exists $\sigma \in \sum \R[\x]^2$ such that $f_\lambda-(f_{{\rm gp},M}-2\epsilon)=\sigma$, i.e., $f-(f_{{\rm gp},M}-2\epsilon)= \sigma+\lambda(M-\sum x_i^{2d})$. It follows that $f_{{\rm sos},M}^{(0)} \ge f_{{\rm gp},M}-2\epsilon$. Since $\epsilon>0$ is arbitrary it follows that $f_{{\rm sos},M}^{(0)} \ge f_{{\rm gp},M}$.
\end{proof}

\begin{rem} \label{remark} \

(1) According to \cite[Cor. 3.4]{gha-mar2}, $|\Omega(f)| = 1$ $\Rightarrow$ $f_{{\rm gp}} = f_{{\rm sos}} = f_*$. The same is true (trivially) if $|\Omega(f)|=0$. Thus if $|\Omega(f)| \le 1$ and $f$ achieves its global minimum in the ball $B_M$ then \[f_* = f_{*,M} \ge f_{{\rm sos},M}^{(0)} \ge f_{{\rm gp},M} \ge f_{{\rm gp}} = f_{{\rm sos}} = f_*,\]
so
\[f_{*,M} = f_{{\rm sos},M}^{(0)} = f_{{\rm gp},M} = f_{{\rm gp}} = f_{{\rm sos}} = f_*.\]

(2) There are explicit formulas for $f_{{\rm gp}}$ and $f_{{\rm gp},M}$ if $|\Delta(f)| =0$. Suppose $|\Delta(f)| =0$. As usual, we suppose that $f_{2d,1}\ge \dots \ge f_{2d,n}$. Then
\[ f_{{\rm gp}} = \begin{cases} f(0) &\text{if } f_{2d,n}\ge 0 \\ -\infty &\text{if } f_{2d,n}<0 \end{cases},\]
and
\[ f_{{\rm gp},M} = \begin{cases} f(0) &\text{if } f_{2d,n}\ge 0 \\ f(0)+Mf_{2d,n} &\text{if } f_{2d,n} < 0 \end{cases}.\]

(3) There are also explicit formulas for $f_{{\rm gp}}$ and $f_{{\rm gp},M}$ if $|\Delta(f)| =1$ and $f_{2d,i}=1$, $i=1,\dots,n$. Suppose that $|\Delta(f)| =1$ and $f_{2d,i}=1$, $i=1,\dots,n$. Let $\Delta(f) = \{ \alpha\}$. There are two cases to consider:

Case (i). Suppose $|\alpha|= 2d$. In this case
\[ f_{{\rm gp}} = \begin{cases} f(0) &\text{if } (\frac{f_{\alpha}}{2d})^{2d} \alpha^{\alpha} \le 1 \\ -\infty &\text{if } (\frac{f_{\alpha}}{2d})^{2d} \alpha^{\alpha} > 1 \end{cases},\]
and
\[ f_{{\rm gp},M} = \begin{cases} f(0) &\text{if } (\frac{f_{\alpha}}{2d})^{2d} \alpha^{\alpha} \le 1 \\ f(0)-M\cdot([(\frac{f_{\alpha}}{2d})^{2d}\alpha^{\alpha}]^{1/2d}-1) &\text{if } (\frac{f_{\alpha}}{2d})^{2d} \alpha^{\alpha} > 1 \end{cases}.\]

Case (ii). Suppose $|\alpha|< 2d$. In this case
\[ f_{{\rm gp}} = f(0)-[2d-|\alpha|][(\frac{f_{\alpha}}{2d})^{2d}\alpha^{\alpha}]^{1/(2d-|\alpha|)},\]
and
\[ f_{{\rm gp},M} = \begin{cases} f(0)-[2d-|\alpha|][(\frac{f_{\alpha}}{2d})^{2d}\alpha^{\alpha}]^{1/(2d-|\alpha|)} &\text{if } M\ge |\alpha|\cdot [(\frac{f_{\alpha}}{2d})^{2d}\alpha^{\alpha}]^{1/(2d-|\alpha|)} \\ f(0)+M-|f_{\alpha}|[(\frac{M}{|\alpha|})^{|\alpha|}\alpha^{\alpha}]^{1/2d} &\text{if } M < |\alpha|\cdot [(\frac{f_{\alpha}}{2d})^{2d}\alpha^{\alpha}]^{1/(2d-|\alpha|)} \end{cases}.\]
\end{rem}

\begin{exm} Suppose $n=1$, $2d=6$, $f = x^6+3x^4-9x^2$. Applying Remark \ref{remark}(3), Case (ii), we see that $f_{{\rm gp}} = -2\cdot 3^{3/2} \approx -10.3923$ and
\[ f_{{\rm gp},M} = \begin{cases} -2\cdot 3^{3/2} &\text{if } M \ge 3^{3/2} \\ M-9M^{1/3} &\text{if } M < 3^{3/2} \end{cases}.\]
In this example one checks that $f_* = -5$, and
\[ f_{*,M} = \begin{cases} -5 &\text{if } M \ge 1 \\ M+3M^{2/3}-9M^{1/3} &\text{if } M < 1 \end{cases}.\]
\end{exm}

\section{Numerical computations}

%\subsection*{Running time efficiency}
To compare the running time efficiency of computation of $f_{{\rm gp},M}$ using geometric
programming with computation of $f^{(0)}_{{\rm sos},M}$ using semidefinite programming, we set up a test over $10$ polynomials for each case to keep track of the running times. The polynomials considered had highest degree part $\sum x_i^{2d}$ with the lower degree coefficients randomly chosen integers between $-10$  and $10$, and  $M$ was taken to be a random integer between
$1$ and $10^5$ (Table \ref{T1})\footnote{\textbf{Hardware and Software specifications.} Processor: Intel\textregistered~ Core\texttrademark2 Duo CPU P8400 @
2.26GHz, Memory: 2 GB, OS: Ubuntu 12.04-32 bit, \textsc{Sage}-4.8}. The source code of the \textsc{Sage} program to compute $f_{{\rm gp}, M}$ and
$f_{{\rm sos}, M}^{(0)}$, developed by the first author, is available at \href{http://goo.gl/iI3Y0}{http://goo.gl/iI3Y0}.
Table \ref{T2} demonstrates the running time efficiency of computing $f_{{\rm gp},M}$ for random polynomials $f$ and random integers $M$ chosen as before but for relatively large
$n$ and $2d$ and with sparsity conditions on the size of $\Omega(f)$.

\begin{table}[ht]
\caption{Average running time for $f_{{\rm gp},M}$ and $f_{{\rm sos},M}^{(0)}$ (seconds)}\label{T1}
\centering
\begin{tabular}{|c|c|cccc|}
\hline
$n$ & $2d$ & 4 & 6 & 8 & 10 \\
\hline
\multirow{2}{*}{3} & $f_{{\rm gp},M}$  & 0.03 & 0.09 & 0.96 & 4.73 \\
				   & $f^{(0)}_{{\rm sos},M}$ & 0.05 & 0.56 & 6.42 & 62.28 \\
\hline
\multirow{2}{*}{4} & $f_{{\rm gp},M}$  & 0.04 & 0.89 & 34.90 & 278.43 \\
				   & $f^{(0)}_{{\rm sos},M}$ & 0.16 & 7.74 & 154.17 & - \\
\hline
\multirow{2}{*}{5} & $f_{{\rm gp},M}$  & 0.10 & 8.25 & 48.28 & 1825.56 \\
				   & $f^{(0)}_{{\rm sos},M}$ & 0.53 & 69.49 & - & - \\
\hline
\end{tabular}
\end{table}
\begin{table}[ht]
\caption{Average running time for $f_{{\rm gp},M}$ (seconds) for various constraints on $|\Omega(f)|$ %calculation with constraint on $|\Omega(f)|$
}\label{T2}
\centering
\begin{tabular}{|c|c|ccccc|}
\hline
$n$ & $2d\backslash|\Omega(f)|$ & 10 & 20 & 30 & 40 & 50 \\
\hline
\multirow{3}{*}{10} & 20 & 0.52 & 0.62 & 1.91 & 4.36 & 5.63 \\
				    & 40 & 0.75 & 1.42 & 2.1 & 5.08 & 11.16 \\
				    & 60 & 0.86 & 1.72 & 3.1 & 6.48 & 13.07 \\
\hline
\multirow{2}{*}{20} & 20 & 3.69 & 18.11 & 17.11 & 44.78 & 46.51 \\
				    & 40 & 3.75 & 18.82 & 37.52 & 59.55 & 114.05 \\
				    & 60 & 7.31 & 27.33 & 46.05 & 96.86 & 164.56 \\
\hline
\multirow{2}{*}{30} & 20 & 3.16 & 19.63 & 34.81 & 44.04 & 175.5 \\
				    & 40 & 6.07 & 22.72 & 105.77 & 217.07 & 315.85 \\
				    & 60 & 13.71 & 72.81 & 132.04 & 453.05 & 667.87 \\
\hline
\multirow{2}{*}{40} & 20 & 6.67 & 37.22 & 63.09 & 131.03 & 481.71 \\
				    & 40 & 11.21 & 76.03 & 83.91 & 458.75 & 504.6 \\
				    & 60 & 24.97 & 114.45 & 355.56 & 796.52 & 1340.76 \\
\hline
\end{tabular}
\end{table}

%\subsection*{Accuracy}
We compare values of
$f_{{\rm gp}, M}$ with corresponding values of $f^{(0)}_{{\rm sos}, M}$ for various choices of $f$ and $M$.% The values of $f_{*,M}$ in these examples were obtained by ad hoc methods.
\begin{exm}
Let $f= w^6 + x^6 + y^6 + z^6 + 7w^4y - 10w^3xy + 5wx^3y- 3w^3y^2 - 3w^2xy^2 + 9wxy^3 - 10xy^4 + 7w^4z +wx^3z - 5xyz^3 - 5z^5 + 8w^4 + 8w^2x^2 - 4wx^3 -w^3y + 2wx^2y + 3w^2y^2 - wxy^2 + wy^3 +7w^2xz - 3y^3z + w^2z^2 + 2y^2z^2 - 2w^3 + 8x^3 -5w^2y + 8x^2z + 3xz - 3z + 5$, then:
\[
\begin{array}{ll}
	f_{{\rm gp},1}\approx -39.022 & f^{(0)}_{{\rm sos}, 1}\approx -5.519 \\
	f_{{\rm gp},10}\approx -213.631 & f^{(0)}_{{\rm sos}, 10}\approx -67.947 \\
	f_{{\rm gp},10^2}\approx -1215.730 & f^{(0)}_{{\rm sos}, 10^2}\approx -489.009 \\
	f_{{\rm gp}}\approx -9580211.794 & f_{{\rm sos}}\approx -458107.262
\end{array}
\]
\end{exm}
\begin{exm}
Let $f= 8w^6 + 6x^6 + 4y^6 + 2z^6 - 3w^3x^2 + 8w^2xyz - 9xz^4 + 2w^2xz - 3xz^2$, then
\[
\begin{array}{ll}
	f_{{\rm gp}, 1}\approx -6.605 & f^{(0)}_{{\rm sos}, 1}\approx -6.605 \\
	f_{{\rm gp}, 10}\approx -27.151 & f^{(0)}_{{\rm sos}, 10}\approx -27.151 \\
	f_{{\rm gp}, 10^2}\approx -73.458 & f^{(0)}_{{\rm sos}, 10^2}\approx -73.458 \\
	f_{{\rm gp}}\approx -74.971 & f_{{\rm sos}}\approx -74.971
\end{array}
\]
\end{exm}
\begin{exm}
For $f= -7x^3y^4 + 13x^2y^5 + 5y^4z + 18xz^4 - 5z^2$ with $2d=8$
\[
\begin{array}{ll}
	f_{{\rm gp}, 1}\approx -23.4559 & f^{(0)}_{{\rm sos}, 1}\approx -19.4797 \\
	f_{{\rm gp}, 10}\approx -117.9727 & f^{(0)}_{{\rm sos}, 10}\approx -92.6547 \\
	f_{{\rm gp}, 10^2}\approx -736.0259 & f^{(0)}_{{\rm sos}, 10^2}\approx -668.221
\end{array}
\]
\end{exm}

We can compute $f_{{\rm gp},M}$ in cases where computation of $f_{{\rm sos},M}^{(0)}$ breaks down.
\begin{exm}
For $f= -9w^{12}x^9y^{12}z^5 + 19w^8x^2yz^{20} - 3w^{11}x^6y^9z^4 - 3w^{13}x^{14}z - 18w^4x^{12}y^3$ with $2d=40$
\[
\begin{array}{l}
	f_{{\rm gp}, 1}\approx -20.0645 \\
	f_{{\rm gp}, 10}\approx -106.4946 \\
	f_{{\rm gp}, 10^2}\approx -584.027
\end{array}
\]
\end{exm}
%\begin{exm}
%For $f=x^{40}+y^{40}+z^{40}+x^{20}y^{7}z^{12} + 6x^{19}y^9z^9 - 9x^7y^{13}z^8 + x^9y^{11}z^7 - x^{10}y^5z^4$, we have
%$f_*=f_{*,10^4}\approx -1477.5781$, the global minimum occurs at $x\approx 1.23064, y\approx 1.21784, z\approx -1.20979$ and $f_{{\rm gp}}\approx -177887.4830$, but $f_{{\rm gp},10^4}\approx -7219.0861$.
%\end{exm}

\begin{exm}
For\\
 $f = \sum_{i=0}^{19} x_i^{20} + x_2^6x_3^3x_5x_7x_8^3x_9x_{10}x_{11}^2x_{12} -
   17x_1x_2x_3x_6x_7x_9^2x_{10}x_{12}^4x_{14}^4x_{16}x_{18}x_{19} +
   19x_4^6x_5^4x_6^2x_9x_{12}x_{17}^2x_{18}x_{19}^2 -
   10x_0x_1^5x_2x_8^3x_{12}x_{15}x_{17}x_{18}^2x_{19}^4 -
   11x_0^2x_2x_4^3x_5x_6x_{12}^4x_{15}^4x_{16}x_{17} +
   15x_1^2x_5^3x_6x_8x_9x_{14}^2x_{15}^4x_{18}^2x_{19}^2 +
   2x_1x_2^2x_4^3x_6x_{10}x_{11}^2x_{13}x_{15}x_{17}x_{18}x_{19}^3$,
\[
\begin{array}{l}
	f_{{\rm gp}, 10}\approx -41.6538 \\
	f_{{\rm gp},10^2}\approx -340.6339 \\
	f_{{\rm gp}, 10^3}\approx -2774.217 \\
	f_{{\rm gp}}\approx -84853211002.07141
\end{array}
\]
\end{exm}

\section{Proof of Theorem \ref{thmain}}
\label{proof}

With $\lambda\geq0$ fixed, let us apply Proposition \ref{prop1} to the polynomial $f_\lambda\in\R[\x]_{2d}$,
so as to obtain a lower bound $\overline{G}(\lambda)$ on $G(\lambda)$
defined in (\ref{Glambda}). Then $\overline{G}(\lambda):=f_\lambda(0)-\rho_\lambda$, with

\[\left\{\begin{array}{rl}
\rho_\lambda=\dis\min_{\z_\alpha}&\dis\sum_{\alpha\in\Delta(f_\lambda)^{<2d}}(2d-\vert\alpha\vert)\left[
\left(\frac{(f_\lambda)_\alpha}{2d}\right)^{2d}\,\left(\frac{\alpha}{\z_\alpha}\right)^\alpha
\right]^{1/(2d-\vert\alpha\vert)}\\
&\\
\mbox{s.t.}&\dis\sum_{\alpha\in\Delta(f_\lambda)}z_{\alpha,i} \le (f_\lambda)_{2d,i},\quad i=1,\ldots,n\\
&\\
&\left(\frac{2d}{(f_\lambda)_\alpha}\right)^{2d}\,\left(\frac{\z_\alpha}{\alpha}\right)^\alpha \,=\,1;\quad\alpha\in\Delta(f_\lambda),\:\vert\alpha\vert=2d.
\end{array}\right.\]
Notice that $\Omega(f_\lambda)=\Omega(f)$, and $(f_\lambda)_\alpha=f_\alpha$ for all $\alpha\in\Omega(f)$. Moreover,
\[f_\lambda(0)\,=\,f(0)-\lambda M;\quad
(f_\lambda)_{2d,i}=f_{2d,i}+\lambda,\:\forall i=1,\ldots,n.\]
And so, with $\lambda\geq0$, $\overline{G}(\lambda):=f(0)-M\lambda-\rho_\lambda$, with
\[\left\{\begin{array}{rl}
\rho_\lambda=\dis\min_{\z_\alpha}&\dis\sum_{\alpha\in\Delta(f)^{<2d}}(2d-\vert\alpha\vert)\left[
\left(\frac{f_\alpha}{2d}\right)^{2d}\,\left(\frac{\alpha}{\z_\alpha}\right)^\alpha
\right]^{1/(2d-\vert\alpha\vert)}\\
&\\
\mbox{s.t.}&\dis\sum_{\alpha\in\Delta(f)}z_{\alpha,i}\,\leq\,f_{2d,i}+\lambda,\quad i=1,\ldots,n\\
&\\
&\left(\frac{2d}{f_\alpha}\right)^{2d}\,\left(\frac{\z_\alpha}{\alpha}\right)^\alpha \,=\,1;\quad\alpha\in\Delta(f),\:\vert\alpha\vert=2d,
\end{array}\right.\]
is a lower bound on $G(\lambda)$ for every $\lambda\geq0$.
Next, recall that
\[f_{2d,1}\,\geq\, f_{2d,2}\,\geq\,\cdots\,\geq\, f_{2d,n}.\]
Let
\[\lambda_0 :=\max\{ 0,-f_{2d,n}\}.\]
For $0 \le \lambda <\lambda_0$, $f_{2d,n}+\lambda <0$, so $\rho_{\lambda}=\infty$, i.e., $\overline{G}(\lambda) =-\infty$. For $\lambda \ge \lambda_0$, $\rho_{\lambda} \le \rho_{\lambda_0}$, i.e., $\overline{G}(\lambda) \ge \overline{G}(\lambda_0)-M(\lambda-\lambda_0)$. Consequently,
\[\max_{\lambda\geq0}\overline{G}(\lambda)=\max_{\lambda\geq \lambda_0}\overline{G}(\lambda) = \max_{\lambda> \lambda_0}\overline{G}(\lambda).\]  For $\lambda >\lambda_0$, using the new variables $u_i:=f_{2d,i}+\lambda>0$, $i=1,\dots,n$, one has:
\[u_i=u_{i-1}-(f_{2d,i-1}-f_{2d,i}),\quad i=2,\ldots,n,\]
or equivalently,
\[\frac{u_i}{u_{i-1}}+\frac{f_{2d,i-1}-f_{2d,i}}{u_{i-1}}\,=\,1,\quad i=2,\ldots,n.\]
In addition, the constraints $\dis\sum_{\alpha\in\Delta}z_{\alpha,i}\leq f_{2d,i}+\lambda$, read
\[\dis\sum_{\alpha\in\Delta}\frac{z_{\alpha,i}}{u_i}\leq 1,\qquad i=1,\ldots,n.\]
Finally, as $\lambda=u_1-f_{2d,1}$, then $f(0)-M\lambda=f(0)+M f_{2d,1}-Mu_1$. Therefore,
for $\lambda> \lambda_0$ fixed, and
\[u_1=\lambda+f_{2d,1};\quad\frac{u_i}{u_{i-1}}+\frac{f_{2d,i-1}-f_{2d,i}}{u_{i-1}}\,=\,1,\quad i=2,\ldots,n,\]
$\overline{G}(\lambda)=f(0)+Mf_{2d,1}-\theta_M(\u)$ with
\[\begin{array}{rl}
\theta_M(\u)=Mu_1+\dis\min_{\z_\alpha}&\dis\sum_{\alpha\in\Delta(f)^{<2d}}(2d-\vert\alpha\vert)\left[
\left(\frac{f_\alpha}{2d}\right)^{2d}\,\left(\frac{\alpha}{\z_\alpha}\right)^\alpha
\right]^{1/(2d-\vert\alpha\vert)}\\
&\\
\mbox{s.t.}&\dis\sum_{\alpha\in\Delta(f)}\frac{z_{\alpha,i}}{u_i}\leq 1,\qquad i=1,\ldots,n\\
&\\
&\left(\frac{2d}{f_\alpha}\right)^{2d}\,\left(\frac{\z_\alpha}{\alpha}\right)^\alpha \,=\,1;\quad\alpha\in\Delta(f),\:\vert\alpha\vert=2d.
\end{array}\]
And so,
\[\max_{\lambda\geq0}\overline{G}(\lambda)=\max_{\lambda\geq \lambda_0}\overline{G}(\lambda) = \max_{\lambda> \lambda_0}\overline{G}(\lambda) =f(0)+Mf_{2d,1}-\rho_M,\]
where
\begin{equation}
\label{aux1}
\begin{array}{rl}
\rho_M=\dis\min_{\z_\alpha,\u}&Mu_1+\dis\sum_{\alpha\in\Delta(f)^{<2d}}(2d-\vert\alpha\vert)\left[
\left(\frac{f_\alpha}{2d}\right)^{2d}\,\left(\frac{\alpha}{\z_\alpha}\right)^\alpha
\right]^{1/(2d-\vert\alpha\vert)}\\
&\\
\mbox{s.t.}&\frac{u_i}{u_{i-1}}+\frac{f_{2d,i-1}-f_{2d,i}}{u_{i-1}}\,=\,1,\quad i=2,\ldots,n\\
&\\
&\frac{f_{2d,1}}{u_1}\leq 1\\
&\\
&\dis\sum_{\alpha\in\Delta(f)}\frac{z_{\alpha,i}}{u_i}\leq 1,\qquad i=1,\ldots,n\\
&\\
&\left(\frac{2d}{f_\alpha}\right)^{2d}\,\left(\frac{\z_\alpha}{\alpha}\right)^\alpha \,=\,1;\quad\alpha\in\Delta(f),\:\vert\alpha\vert=2d.
\end{array}\end{equation}

Notice that (\ref{aux1}) is not a geometric program because of the presence of
$n-1$ posynomial {\it equality} constraints.
To obtain a geometric program, observe that in (\ref{aux1}) we can relax the $n-1$ posynomial equality constraints
\begin{equation}
\label{aux2}
\frac{u_i}{u_{i-1}}+\frac{f_{2d,i-1}-f_{2d,i}}{u_{i-1}}\,=\,1,\quad i=2,\ldots,n,\end{equation}
to the posynomial {\it inequality} constraints
\begin{equation}
\label{aux3}
\frac{u_i}{u_{i-1}}+\frac{f_{2d,i-1}-f_{2d,i}}{u_{i-1}}\,\leq\,1,\quad i=2,\ldots,n,\end{equation}
without changing the optimal value. Indeed, suppose that $\u$ is an optimal solution
of (\ref{aux1}) with (\ref{aux3}) in lieu of (\ref{aux2}). Then
increase $u_2$ to $u'_2:=u_2+\delta_2$ with $\delta_2>0$ so that
\[\frac{u_2+\delta_2}{u_{1}}+\frac{f_{2d,1}-f_{2d,2}}{u_{1}}\,=\,1.\]
Since $0<u_2\leq u'_2$, the constraint $\dis\sum_{\alpha\in\Delta}\frac{z_{\alpha,2}}{u'_2}\leq 1$
and the constraint $\frac{u_3}{u'_{2}}+\frac{f_{2d,2}-f_{2d,3}}{u'_{2}}\,\leq\,1$,
are satisfied. Therefore, one may repeat the process now with $u_3$, i.e., increase $u_3$ to $u'_3=u_3+\delta_3$ with $\delta_3$ so that
\[\frac{u_3+\delta_3}{u'_{2}}+\frac{f_{2d,2}-f_{2d,3}}{u'_{2}}\,=\,1.\]
Since $0<u_3\leq u'_3$, the constraint $\dis\sum_{\alpha\in\Delta}\frac{z_{\alpha,3}}{u'_3}\leq 1$
and the constraint $\frac{u_4}{u'_{3}}+\frac{f_{2d,3}-f_{2d,4}}{u'_{3}}\,\leq\,1$, are satisfied, etc. Iterate the process to finally obtain a feasible solution
$((z_{\alpha,i}),\u')$ for (\ref{aux1}), with  the desired property. In addition, since $u_1$ and $(z_{\alpha,i})$ have not been changed,
the cost associated to the new feasible solution $((z_{\alpha,i}),\u')$  is the same.
\qed


\begin{thebibliography}{las}
\bibitem{boyd}
S. Boyd, S.-J. Kim, L. Vandenberghe, A. Hassibi. {\em A tutorial on geometric programming},
Optim. Eng. {\bf 8} (2007), pp. 67--127.
\bibitem{fidalgo}
C. Fidalgo, A. Kovacek. {\em Positive semidefinite diagonal minus tail forms are sums of squares}, Math. Z. {\bf 269} (2010), pp. 629--645.
\bibitem{lasserre1}
J.B. Lasserre. {\em Global Optimization with Polynomials and the Problem of Moments}, SIAM J. Optim. {\bf 11}(3) (796-817), 2001.
\bibitem{lasserre}
J.B. Lasserre. {\em Sufficient conditions for a real polynomial to be a sum of squares}, Arch. Math. (Basel) {\bf 89} (2007), pp. 390--398.
\bibitem{gha-mar1}
M. Ghasemi, M. Marshall. {\em Lower bounds for a polynomial in terms of its coefficients}, Arch. Math. (Basel) {\bf 95} (2010), pp. 343--353.
\bibitem{gha-mar2}
M. Ghasemi, M. Marshall. {\em Lower bounds for polynomials using geometric programming}, SIAM J. Optim. {\bf 22}(2) (2012), pp. 460--473.
\bibitem{kim}
H. Waki, S. Kim, M. Kojima, M. Muramatsu. {\em Sums of squares and semidefinite programming relaxations for polynomial optimization problems with structured sparsity}, SIAM J. Optim. {\bf 17}(1) (2006), pp. 218--242.
\end{thebibliography}
\end{document}